\documentclass{amsart}
\usepackage{amsmath, amsthm, amscd, amsfonts, amssymb, graphicx, color}
\usepackage{enumerate}
\usepackage[bookmarksnumbered, colorlinks, plainpages]{hyperref}
\usepackage{tikz}
\usepackage{comment}

\newtheorem{theorem}{Theorem}[section]
\newtheorem{lemma}[theorem]{Lemma}
\newtheorem{proposition}[theorem]{Proposition}

\newtheorem{corollary}[theorem]{Corollary}
\newtheorem{definition}[theorem]{Definition}
\newtheorem{example}[theorem]{Example}
\newtheorem{remark}[theorem]{Remark}

\newcommand{\R}{\mathbb{R}}

\newcommand{\N}{\mathbb{N}}

\excludecomment{mysection}

\begin{document}
	
\title[A generalization of spin factors]{A generalization of spin factors}
\author{Anil Kumar Karn}
	
\address{School of Mathematical Sciences, National Institute of Science Education and Research, HBNI, Bhubaneswar, P.O. - Jatni, District - Khurda, Odisha - 752050, India.}

\email{\textcolor[rgb]{0.00,0.00,0.84}{anilkarn@niser.ac.in}}

	
	
\begin{abstract}
	Using a technique of adjoining an order unit to a normed linear space, we have characterized strictly convex spaces among normed linear spaces and Hilbert spaces among strictly convex Banach spaces respectively. This leads to a generalization of spin factors and provides a new class of absolute order unit spaces.  
\end{abstract}

\maketitle 

\section{Introduction}
Let $H$ be a real Hilbert space, $A$ be a C$^*$-algebra and let $\chi: H \to A$ be linear mapping. Then the set of conditions:
\begin{enumerate}
	\item $2 \chi(\xi) \circ \chi(\eta) := \chi(\xi) \chi(\eta) + \chi(\eta) \chi(\xi) = 0$; and 
	\item $2 \chi(\xi) \circ \chi(\eta)^* := \chi(\xi) \chi(\eta)^* + \chi(\eta)^* \chi(\xi) = \langle \xi, \eta \rangle 1$
\end{enumerate}
for all $\xi, \eta \in H$ is called a canonical anticommutation relation (CAR) on $A$. The CARs play an important role in mathematical physics. Let us put $\chi(\xi) + \chi(\xi)^* =  \psi(\xi)$ for each $\xi \in H$. Then $\psi:H \to A_{sa}$ is also an isometric linear mapping satisfying 
\begin{enumerate}
	\item[(3)] $\psi(\xi) \circ \psi(\eta) = \langle \xi, \eta \rangle 1$ 
\end{enumerate} 
for all $\xi, \eta \in H$. Also, $\psi(H) + \mathbb{R} 1$ is a unital $JC$-algebra of $A_{sa}$. (See, for example \cite[6.1.1]{H-OS84}.) Conversely, if we define 
\begin{enumerate}
	\item[(4)] $(\xi, \alpha) \circ (\eta, \beta) = (\beta \xi + \alpha \eta, \langle \xi, \eta \rangle + \alpha \beta)$
\end{enumerate}
for $\xi, \eta \in H$ and $\alpha, \beta \in \mathbb{R}$, then $H \oplus \mathbb{R}$ is isometrically Jordan isomorphic to a $JC$-algebra. For details, please refer to \cite[Chapter 6]{H-OS84}. Such a space is called a spin factor. In this paper, we discuss an order theoretic generalization of spin factors.

A normed linear space is said to be strictly convex, if the line segment joining any two points on its unit sphere does not meet the unit sphere except for its extremities. The class of strictly convex spaces include $\ell_p$-spaces as well as $L_p$-spaces for $1 < p < \infty$.  It is an important class of normed linear spaces and enjoys many geometric properties. A detailed study on this topic can be found in several books, see, for example \cite{Dies75} as a good source of information. In what follows, we shall explore an order theoretic aspect of strict convexity.

The notion of absolute order unit space was introduced by the author in \cite{K18} to propose a non-commutative model for vector lattices. Its examples include the self-adjoint part of a unital C$^*$-algebras or more generally, unital JB-algebras. In particular, a spin factor is an absolute order unit space. In this paper, we describe a new class of absolute order unit spaces by adjoining an order unit to a strictly convex space which is a generalization of spin factors.

It is a folklore that an order unit can be adjoined to a normed linear space resulting in an order unit space. The following construction has been adopted from \cite[1.6.1]{Jam70} and is apparently due to M. M. Day. Let $(V, \Vert \cdot \Vert)$ be a real normed linear space. Consider $V^{(\cdot)} := V \times \mathbb{R}$ and define 
$$V^{(\cdot)+} := \lbrace (v, \alpha): \Vert v \Vert \le \alpha \rbrace.$$ 
Then $(V^{(\cdot)}, V^{(\cdot)+})$ becomes a real ordered space such that $V^{(\cdot)+}$ is proper, generating and Archimedean. Also, $e = (0, 1) \in V^{(\cdot)+}$ is an order unit for $V^{(\cdot)}$ so that $(V^{(\cdot)}, e)$ becomes an order unit space.  The corresponding order unit norm is given by 
$$\Vert (v, \alpha) \Vert = \Vert v \Vert + \vert \alpha \vert$$
for all $(v, \alpha) \in V^{(\cdot)}$. Thus $V$, identified with $\lbrace (v, 0): v \in V \rbrace$, can be identified as a closed subspace of $V^{(\cdot)}$. Further, $V^{(\cdot)}$ is complete if and only if so is $V$. 

In this paper, we describe the notion of an absolute value in $V^{(\cdot)}$ which arises naturally from the definition of $V^{(\cdot)+}$. We prove that the absolute value satisfies all the conditions to confirm $V^{(\cdot)}$ as an absolutely ordered space (definition is given below) provided that $V$ is strictly convex. Further, for each $t$, $1 \le t \le \infty$, we introduce a norm on $V^{(\cdot)}$ so that it becomes an absolute order smooth $t$-normed space. For $t = \infty$, $V^{(\infty)}$ becomes an absolute order unit space (Theorem \ref{10a}) and for $t = 1$, $V^{(1)}$ becomes an absolutely base normed space (Theorem \ref{11}). 

Next, we focus on the absolute order unit space $(V^{(\infty)}, e)$ where $V$ is a strictly convex Banach space. We discuss its spectral property and prove that any element $(v, \alpha) \in V^(\infty)$ has the spectral decomposition of the form
$$(v, \alpha) = (\alpha - \Vert v \Vert) \left( \frac{- v}{2 \Vert v \Vert}, \frac 12 \right) + (\alpha + \Vert v \Vert) \left( \frac{v}{2 \Vert v \Vert}, \frac 12 \right).$$ 
This observation leads to a notion of the (positive) integral powers of $(v, \alpha)$ and consequently to the notion of $\circ$ in $V^{({\infty})}$. We prove that $\circ$ is bilinear in $V^{({\infty})}$ if and only if $V$ is a Hilbert space. As a consequence, we observe that $V^{({\infty})}$ is a generalization of spin factors.

Now, we sketch a summary of the paper.

In Section 2, we discuss the properties of an absolute value on $V$ which naturally arises out of the definition of $V_0^{(\cdot)+}$. Further, we prove that it is necessary as well as sufficient that $V_0$ must be strictly convex in order to $V_0^{(\cdot)}$ be an absolutely ordered space. We further show that $(V^{(\cdot)}, e)$ is an absolute order unit space. This is a new example of an absolute order unit space. 

In Sections 3, we discuss the spectral property in the absolute order unit space $V^{({\infty})}$ to introduce a binary operation in it. We show that this binary operation is bilinear if and only if $V$ is a Hilbert space.

\section{The order structure}

In this section, we continue with an arbitrarily fixed real normed linear space $V$ and the corresponding real ordered space $(V^{(\cdot)}, V^{(\cdot)+})$ where $V^{(\cdot)} = V \times \mathbb{R}$ and $V^{(\cdot)+} = \lbrace (v, \alpha): \Vert v \Vert \le \alpha \rbrace$. We also keep the order unit $e = (0, 1) \in V^{(\cdot)+}$. These notations are fixed throughout the paper, unless stated otherwise. 

The following observation is also going to be handy throughout the paper.
\begin{enumerate}
	\item $(v, \alpha) \in V^{(\cdot)+}$ if and only if $\Vert v \Vert \le \alpha$; 
	\item $(v, \alpha) \in - V^{(\cdot)+}$ if and only if $\Vert v \Vert \le - \alpha$; and 
	\item $(v, \alpha) \notin V^{(\cdot)+} \cup - V^{(\cdot)+}$ if and only if $\vert \alpha \vert < \Vert v \Vert$.
\end{enumerate}
Based on this observation, we propose the following: 
\begin{definition}
	For $(v, \alpha) \in V^{(\cdot)}$, we define 
	\begin{displaymath}
	\vert (v, \alpha) \vert = \begin{cases}
	(v, \alpha),  &\textrm{if} ~ (v, \alpha) \in V^{(\cdot)+} \\
	- (v, \alpha), &\textrm{if} ~ (v, \alpha) \in - V^{(\cdot)+} \\
	\left(\frac{\alpha}{\Vert v \Vert} v, \Vert v \Vert \right), &\textrm{if} ~ (v, \alpha) \notin V^{(\cdot)+} \cup - V^{(\cdot)+}.
	\end{cases}
	\end{displaymath}
	Then $\vert \cdot \vert: V^{(\cdot)} \to V^{(\cdot)+}$ is called an \emph{absolute value} in $V^{(\cdot)}$.
\end{definition}
Let us recall the following term introduced in \cite{K18}. (See also, \cite{K16}.)
\begin{definition}\cite[Definition 3.4]{K18}
	Let $(U, U^+)$ be a real ordered vector space and let $\vert\cdot\vert: U \to U^+$ satisfy the following conditions:               
	\begin{enumerate}
		\item $\vert v \vert = v$ if $v \in U^+$;
		\item $\vert v \vert \pm v \in U^+$ for all $v \in U$;
		\item $\vert k v \vert = \vert k \vert \vert v \vert$ for all $v \in U$ and $k \in \mathbb{R}$; 
		\item If $u, v, w \in U$ with $\vert u - v \vert = u + v$ and $\vert u - w \vert = u + w,$ then $\vert u - \vert v \pm w \vert \vert = u + \vert v \pm w \vert$; 
		\item If $u, v$ and $w \in U$ with $\vert u - v \vert = u + v$ and $0 \le w \le v,$ then $\vert u - w \vert = u + w$.
	\end{enumerate}                 
	Then $(U, U^+, \vert\cdot\vert)$ is called a \emph{absolutely ordered space}. 
\end{definition}
In this section, we show that $(V^{(\cdot)}, V^{(\cdot)+}, \vert \cdot \vert)$ is an absolutely ordered space. 
\begin{proposition}\label{01}
	For $(v, \alpha) \in V^{(\cdot)}$ and $k \in \mathbb{R}$, we have 
	\begin{enumerate}[$(a)$]
		\item $\vert (v, \alpha) \vert = (v, \alpha)$ if $(v, \alpha) \in V^{(\cdot)+}$; 
		\item $\vert (v, \alpha) \vert \pm (v, \alpha) \in V^{(\cdot)+}$; and 
		\item $\vert k (v, \alpha) \vert = \vert k \vert \vert (v, \alpha) \vert$.  
	\end{enumerate}
\end{proposition} 
\begin{proof}
	Verification of (a) is straight forward. Further, (b) and (c) also may be verified easily if $(v, \alpha) \in V^{(\cdot)+} \bigcup - V^{(\cdot)+}$. So we assume that $(v, \alpha) \notin V^{(\cdot)+} \bigcup - V^{(\cdot)+}$. Then $\vert \alpha \vert < \Vert v \Vert$ and $\vert (v, \alpha) \vert = \left( \frac{\alpha}{\Vert v \Vert}v, \Vert v \Vert \right)$. Thus 
	$$\vert (v, \alpha) \vert + (v, \alpha) = \left( \frac{(\alpha + \Vert v \Vert)}{\Vert v \Vert} v, \Vert v \Vert + \alpha \right)$$ 
	and 
	$$\vert (v, \alpha) \vert - (v, \alpha) = \left( \frac{(\alpha - \Vert v \Vert)}{\Vert v \Vert} v, \Vert v \Vert - \alpha \right).$$ 
	Since 
	$$\left\Vert \frac{(\alpha + \Vert v \Vert)}{\Vert v \Vert} v \right\Vert = \Vert v \Vert + \alpha$$ 
	and 
	$$\left\Vert \frac{(\alpha - \Vert v \Vert)}{\Vert v \Vert} v \right\Vert = \Vert v \Vert - \alpha,$$ 
	we again get that $\vert (v, \alpha) \vert \pm (v, \alpha) \in V^{(\cdot)+}$. Next, 
	$$\vert k(v, \alpha) \vert = \left( \frac{(k \alpha)}{\Vert k v \Vert} (k v), \Vert k v \Vert \right) = \vert k \vert \left( \frac{\alpha}{\Vert v \Vert} v, \Vert v \Vert \right) = \vert k \vert \vert (v, \alpha).$$ 
	Thus (b) and (c) also hold in all the three cases. 
\end{proof}
To prove the other conditions, we need to recall the following notion from \cite{K18}.
\begin{definition}
	For $(u, \alpha), (v, \beta) \in V_0^{(\cdot)}$, we say that $(u, \alpha)$ is \emph{orthogonal} to $(v, \beta)$ (we write, $(u, \alpha) \perp (v, \beta)$), if $\vert (u, \alpha) - (v, \beta) \vert = (u, \alpha) + (v, \beta)$.
\end{definition}
\begin{remark}
	\begin{enumerate}[$(i)$]
		\item $(u, \alpha) \perp 0$ for all $(u, \alpha) \in V^{(\cdot)+}$.
		\item If $(u, \alpha) \perp (v, \beta)$, then by Proposition \ref{01}(b), we have $(u, \alpha), (v, \beta) \in V^{(\cdot)+}$. 
		\item If $(u, \alpha) \perp e$, then $(u, \alpha) = 0$. 
	\end{enumerate}	
\end{remark}
\begin{proposition}\label{02}
	Let $(u, \alpha) \perp (v, \beta)$ with $(u, \alpha), (v, \beta) \in V^{(\cdot)+} \setminus \lbrace 0 \rbrace$. Then $\alpha = \Vert u \Vert > 0$, $\beta = \Vert v \Vert > 0$ and $\frac{u}{\Vert u \Vert} = - \frac{v}{\Vert v \Vert}$. Conversely, let $u \in V$ with $\Vert u \Vert = 1$. Then $(\alpha u, \alpha ) \perp (- \beta u, \beta)$ for all $\alpha, \beta \in \mathbb{R}^+$.
\end{proposition}
\begin{proof}
	First, we assume that $(u, \alpha) \perp (v, \beta)$ with $(u, \alpha), (v, \beta) \in V^{(\cdot)+} \setminus \lbrace 0 \rbrace$. As $\alpha = 0$ forces $\Vert u \Vert = 0$ for $(u, \alpha) \in V^{(\cdot)+}$, we must have $\alpha > 0$. For the same reason, we have $\beta > 0$ as well. Further, if $(u, \alpha) - (v, \beta) \in V^{(\cdot)+}$, then as $(u, \alpha) \perp (v, \beta)$, we have $(u, \alpha) + (v, \beta) = \vert (u, \alpha) - (v, \beta) \vert$. But then, by Proposition \ref{01}(a), we shall get $(v, \beta) = 0$ so that $(u, \alpha) - (v, \beta) \notin V^{(\cdot)+}$. In the same way, we can also show that $(u, \alpha) - (v, \beta) \notin - V^{(\cdot)+}$. Therefore, $\vert \alpha - \beta \vert < \Vert u - v \Vert$ and it follows that 
	$$\vert (u, \alpha) - (v, \beta) \vert = \left(\frac{(\alpha - \beta)}{\Vert u - v \Vert} (u - v), \Vert u - v \Vert \right).$$ 
	Since $(u, \alpha) + (v, \beta) = \vert (u, \alpha) - (v, \beta) \vert$, we deduce that $u + v = \frac{(\alpha - \beta)}{\Vert u - v \Vert} (u - v)$ and $\alpha + \beta = \Vert u - v \Vert$. Thus $\alpha v + \beta u = 0$ so that $\alpha \Vert v \Vert = \beta \Vert u \Vert$. Also, as $\alpha + \beta = \Vert u - v \Vert$, we further have 
	$$\alpha + \beta = \left\Vert u + \frac{\beta}{\alpha} u \right\Vert = \frac{(\alpha + \beta)}{\alpha} \Vert u \Vert$$ 
	so that $\Vert u \Vert = \alpha$ and consequently, $\Vert v \Vert = \beta$. 
	
	Conversely, assume that $u \in V$ with $\Vert u \Vert = 1$ and let $\alpha, \beta \in \mathbb{R}^+$. For definiteness, we let $\alpha \ge \beta > 0$ as $\alpha = 0$ and $\beta = 0$ are trivial cases. Since 
	$$\alpha + \beta > \alpha - \beta = \vert \alpha - \beta \vert,$$ 
	we have 
	$$\vert (\alpha u, \alpha ) - (- \beta u, \beta) \vert = \left( (\alpha - \beta) u, \alpha + \beta \right) = (\alpha u, \alpha ) + (- \beta u, \beta).$$ 
	Thus $(\alpha u, \alpha ) \perp (- \beta u, \beta)$.
\end{proof}
\begin{corollary}\label{03}
	Let $(u, \alpha), (v, \beta), (w, \gamma) \in V^{(\cdot)+} \setminus \lbrace 0 \rbrace$ be such that $(u, \alpha) \perp (v, \beta)$ and $(u, \alpha) \perp (w, \gamma)$. Then $(u, \alpha) \perp \vert (v, \beta) \pm (w, \gamma) \vert$.
\end{corollary}
\begin{proof}
	By Proposition \ref{02}, we have $\alpha = \Vert u \Vert > 0$, $\beta = \Vert v \Vert > 0$, $\gamma = \Vert w \Vert > 0$ and $- \frac{u}{\Vert u \Vert} = \frac{v}{\Vert v \Vert} = \frac{w}{\Vert w \Vert}$. Thus 
	$$(v, \beta) \pm (w, \gamma) = (\beta \pm \gamma) \left(-  \frac{u}{\Vert u \Vert}, 1 \right)$$ 
	so that 
	$$\vert (v, \beta) \pm (w, \gamma) \vert = \vert \beta \pm \gamma \vert \left(-  \frac{u}{\Vert u \Vert}, 1 \right).$$ 
	Now, again applying Proposition \ref{02}, we get $(u, \alpha) \perp \vert (v, \beta) \pm (w, \gamma) \vert$.
\end{proof}
We consider a special kind of orthogonal pair introduced in \cite{K18}.
\begin{definition}
	Let $(0, 0) \le (u, \alpha) \le (0, 1)$. We say that $(u, \alpha)$ is an \emph{order projection}, if $(u, \alpha) \perp (- u, 1 - \alpha)$. The set of all order projections in $V^{(\cdot)}$ is denoted by $OP(V^{(\cdot)})$.
\end{definition} 
We know the complete description of $OP(V^{(\cdot)})$.
\begin{proposition}
	$OP(V^{(\cdot)}) = \left\lbrace \left(u, \frac{1}{2} \right): \Vert u \Vert = \frac{1}{2} \right\rbrace \cup \left\lbrace 0, e \right\rbrace$.
\end{proposition}
\begin{proof}
	Let $(u, \alpha) \in OP(V^{(\cdot)})$. If $u = 0$, then $(0, \alpha) \perp (0, 1 - \alpha)$ so that either $\alpha = 0$ or $\alpha = 1$. Now let $u \ne 0$. Then by Proposition \ref{02}, we have $\Vert u \Vert = \alpha = 1 - \alpha$. Thus $\alpha = \frac 12 = \Vert u \Vert$. 
	
%
%
%
	
	The converse can be proved in a routine way by using Proposition \ref{02}. 
\end{proof}
\begin{corollary}\label{04}
	Let $(u, \alpha), (v, \beta) \in V^{(\cdot)+} \setminus \lbrace 0 \rbrace$. Then $(u, \alpha) \perp (v, \beta)$ if and only if there exists a unique $p \in OP(V^{(\cdot)}) \setminus \lbrace 0, e \rbrace$ and $\lambda, \mu > 0$ such that $(u, \alpha) = \lambda p$ and $(v, \beta) = \mu (e - p)$.
\end{corollary}
\begin{proof}
	By Proposition \ref{02}, we have $\alpha = \Vert u \Vert > 0$, $\beta = \Vert v \Vert > 0$ and $\frac{u}{\Vert u \Vert} = - \frac{v}{\Vert v \Vert}$. We put $u_0 = \frac{u}{2 \Vert u \Vert}$, $p = \left( u_0, \frac{1}{2} \right)$, $\lambda = 2 \alpha$ and $\mu = 2 \beta$. Then $p \in OP(V^{(\cdot)})$ and we have $(u, \alpha) = \lambda p$ and $(v, \beta) = \mu (e - p)$. The converse follows directly from Proposition \ref{02}. 
\end{proof}
\begin{remark}\label{05}
	Let $(v, \alpha) \in V^{(\cdot)}$, $v \ne 0$. Then there exists unique $p \in OP(V^{(\cdot)})$ such that 
	$$(v, \alpha) = (\alpha + \Vert v \Vert) p + (\alpha - \Vert v \Vert) (e - p)$$ 
	in the following sense: When $v \ne 0$, we have $p = \left(\frac{v}{2 \Vert v \Vert}, \frac{1}{2}\right)$ and when $v = 0$, we have $p = e$. [The $v = 0$ case may appear undecided, if we notice that as $e = p + (e - p)$ for all $p \in OP(V^{(\cdot)})$ and as $v = 0$, any $p \in OP(V^{(\cdot)})$ works. But then, as the end form is the same, we formally consider the said form for definiteness.]
	
	Further, if $(v, \alpha) \in V^{(\cdot)+}$, then $\alpha + \Vert v \Vert \ge \alpha - \Vert v \Vert \ge 0$. 
	
	In general, for any $(v, \alpha)$, we have 
	$$\vert (v, \alpha) \vert = \vert \alpha + \Vert v \Vert \vert p + \vert \alpha - \Vert v \Vert \vert (e - p).$$
\end{remark}

\subsection{Strictly convex spaces} 
The following result is the backbone of the paper. 
\begin{lemma}\label{06l}
	$V$ is strictly convex if and only of the following condition holds: for $u, v \in V$ with $\Vert u \Vert = 1$, we have $v = \alpha u$ whenever $(0, 0) \le (v, \alpha) \le (u, 1)$ in $V^{(\cdot)+}$ for some $\alpha \in \R$.
\end{lemma} 
\begin{proof}
	First, we assume that $V$ is strictly convex. Let $u, v \in V$ with $\Vert u \Vert = 1$ be such that $(0, 0) \le (v, \alpha) \le (u, 1)$ in $V^{(\cdot)+}$ for some $\alpha \in \R$. We show that $v = \alpha u$. By Hahn-Banach theorem, there exists $f \in V^*$ with $\Vert f \Vert = 1$ such that $f(u) = - 1$. We define $g: V^{(\cdot)} \in \R$ given by $g(w, \gamma) = f(w) + \gamma$ for all $w \in V$ and $\gamma \in \R$. Then $g$ is linear with $g(e) = 1$. If $(x, k) \in V^{(\cdot)+}$, then 
	$$k \ge \Vert x \Vert \ge \vert f(x) \vert \ge  - f(x)$$ 
	so that $g(x, k) = f(x) + k \ge 0$. Thus $g$ is a positive linear functional on $V^{(\cdot)}$. (In fact, $g$ is a state of $V^{(\cdot)}$.) Also, by construction, $g(u, 1) = 0$ so that $g(v, \alpha) = 0$ as $(0, 0) \le (v, \alpha) \le (u, 1)$. Thus we have $\alpha = - f(v) \le \Vert v \Vert \le \alpha$ as $(v, \alpha) \in V^{(\cdot)+}$. That is, $\Vert v \Vert = \alpha$. Again, as $(0, 0) \le (u - v, 1 - \alpha) \le (u, 1)$, we may also conclude that $\Vert u - v \Vert = 1 - \alpha$. Thus 
	$$\Vert v \Vert + \Vert u - v \Vert = 1 = \Vert u \Vert.$$ 
	If $v = 0$, then $\alpha = 0$ and we have $v = 0 u$. Similarly, if $v = u$, then $\alpha = 1$ and we have $v = 1 u$. So we may assume that $v \ne 0$ and $v \ne u$. Since $V$ is strictly convex, we must have $\frac{v}{\Vert v \Vert} = \frac{u - v}{\Vert u - v \Vert}$. Simplifying, we get $v = \alpha u$. 
	
	Next, we assume that $V$ is not strictly convex. Then we can find $u, v \in V$, $u \ne v$ with $\Vert u \Vert = 1 = \Vert v \Vert$ such that $\Vert u + v \Vert = 2$. Put $2 w = u + v$. Then $(\frac 12 u, \frac 12), (w, 1) - (\frac 12 u, \frac 12) = (\frac 12 v, \frac 12) \in V^{(\cdot)+}$ but $(\frac 12 u, \frac 12) \ne \frac 12 (w, 1)$. 
\end{proof}
\begin{corollary}\label{06}
	$V$ is strictly convex if and only of the following condition hold: for $(u, \alpha), (v, \beta), (w, \gamma) \in V^{(\cdot)+} \setminus \lbrace 0 \rbrace$ with $(u, \alpha) \perp (v, \beta)$ and $(w, \gamma) \le (v, \beta)$, we have $(u, \alpha) \perp (w, \gamma)$.
\end{corollary}
\begin{proof}
	First, we assume that $V$ is strictly convex. Let $(u, \alpha), (v, \beta), (w, \gamma) \in V^{(\cdot)+} \setminus \lbrace 0 \rbrace$ such that $(u, \alpha) \perp (v, \beta)$ and $(w, \gamma) \le (v, \beta)$. Then by Proposition \ref{02}, we have $\alpha = \Vert u \Vert > 0$, $\beta = \Vert v \Vert > 0$ and $\frac{u}{\Vert u \Vert} = - \frac{v}{\Vert v \Vert}$. Also, then by Lemma \ref{06l}, we have $\frac{w}{\Vert w \Vert} = \frac{v}{\Vert v \Vert} = - \frac{u}{\Vert u \Vert}$. Hence $(u, \alpha) \perp (w, \gamma)$ by Proposition \ref{02}.

	Conversely, assume that $V$ is not strictly convex. Then by Lemma \ref{06l}, we can find $u, v \in V$ with $\Vert u \Vert = 1$ and $\alpha \in \R$ such that $(0, 0) \le (v, \alpha) \le (u, 1)$ in $V^{(\cdot)+}$ but $v \ne \alpha u$. Thus by Proposition \ref{02}, $( - u, 1) \perp (u, 1)$ but $(- u, 1) \not\perp (v, \alpha)$. 
\end{proof} 
Now Proposition \ref{01}, Corollary \ref{03} and Theorem \ref{06} assimilate into the following: 
\begin{theorem}\label{07}
	Let $V$ be a real normed linear space. Consider $V^{(\cdot)} := V \times \mathbb{R}$ and put $V^{(\cdot)+} := \lbrace (v, \alpha): \Vert v \Vert \le \alpha \rbrace$. Then $(V^{(\cdot)}, V^{(\cdot)+})$ becomes a real ordered space. For $(v, \alpha) \in V^{(\cdot)}$, we define 
	\begin{displaymath}
	\vert (v, \alpha) \vert = \begin{cases}
	(v, \alpha), &\textrm{if} ~ (v, \alpha) \in V^{(\cdot)+} \\
	- (v, \alpha), &\textrm{if} ~ (v, \alpha) \in - V^{(\cdot)+} \\
	\left(\frac{\alpha}{\Vert v \Vert}v, \Vert v \Vert \right), &\textrm{if} ~ (v, \alpha) \notin V^{(\cdot)+} \bigcup - V^{(\cdot)+}.
	\end{cases}
	\end{displaymath}
	Then $(V^{(\cdot)}, V^{(\cdot)+}, \vert \cdot \vert)$ is an absolutely ordered space if and only if $V$ is strictly convex. 
\end{theorem}
\begin{remark}
	Given a normed linear space $V$, we say that the absolutely ordered space $(V^{(\cdot)}, V^{(\cdot)+}, \vert \cdot \vert)$ is obtained by adjoining an order unit to $V$. 
\end{remark}


Let $(U, U^+, \Vert \cdot \Vert)$ be a normed ordered linear space. Then a pair of positive elements $u, v \in U^+$ is said to be \emph{absolutely $\infty$-orthogonal} (we write, $u \perp_{\infty}^a v$), $1 \le t \le \infty$, if for $0 \le u_1 \le u$ and $0 \le v_1 \le v$ we have 
$$\Vert u_1 + k v_1 \Vert = \max \left( \Vert u_1 \Vert, \Vert k v_1 \Vert \right)$$
for all $k \in \mathbb{R}$ (that is, $u_1$ is $\infty$-orthogonal to $v_1$ (we write,$u_1 \perp_{\infty} v_1$)) \cite{K16}. 
\begin{definition}\cite{K16, K18}
	Let $(U, U^+, \vert \cdot \vert)$ be an absolutely ordered space and let $\Vert \cdot \Vert$ be a norm on $V$. Then $(U, U^+, \vert \cdot \vert, \Vert \cdot \Vert)$ is said to be an \emph{absolute order smooth $\infty$-normed space}, if it satisfies the following conditions:
	\begin{enumerate}
		\item[$(O. \infty. 1)$:] For $u \le v \le w$ in $U$, we have $\Vert v \Vert = \max \left( \Vert u \Vert, \Vert w \Vert \right)$; 
		\item[$(O. \perp_{\infty}. 1)$:] if $u, v \in U^+$ with $u \perp v$, then $u \perp_{\infty}^a v$; and
		\item[$(O. \perp_{\infty}. 2)$:] if $u, v \in U^+$ with $u \perp_{\infty}^a v$, then $u \perp v$.
	\end{enumerate}
	An absolute order smooth $\infty$-normed space $(U, U^+, \vert \cdot \vert, \Vert \cdot \Vert)$ is said to be an \emph{absolute order unit space}, if there exists a order unit $e \in U^+$ for $U$ which determine $\Vert \cdot \Vert$ as an order unit norm.
\end{definition}
\begin{theorem}\label{10a}
	Let $V$ be a strictly convex real normed linear space. Then $(V^{(\cdot)}, e)$ is an absolute order unit space.
\end{theorem}
\begin{proof}
	By Theorem \ref{07}, $V^{(\cdot)}$ is an absolutely ordered space. Also, $(V^{(\cdot)}, e)$ is an order unit space so that it satisfies $(O. \infty. 1)$.
	
	Next, we show that for $(u, \alpha), (v, \beta) \in V^{(\cdot)+}$, $(u, \alpha) \perp (v, \beta)$ implies $(u, \alpha) \perp_{\infty}^a (v, \beta)$. As $(0, 0) \perp_{\infty}^a (w, \gamma)$ for all $(w, \gamma) \in V^{(\cdot)+}$, we assume that $(u, \alpha) \ne (0, 0)$ and $(v, \beta) \ne (0, 0)$. Since $(u, \alpha) \perp (v, \beta)$, by Corollary \ref{04}, we have $(u, \alpha) = 2 \alpha (u_0, \frac{1}{2})$ and $(v, \beta) = 2 \beta (- u_0, \frac{1}{2})$ where $u_0 = \frac{u}{2 \Vert u \Vert}$. Let $(0, 0) \le (u_1, \alpha_1) \le (u, \alpha)$ and $(0, 0) \le (v_1, \beta_1) \le (v, \beta)$. Then as in the proof of Proposition \ref{06}, we have $(u_1, \alpha_1) = 2 \alpha_1 (u_0, \frac{1}{2})$ and $(v_1, \beta_1) = 2 \beta_1 (- u_0, \frac{1}{2})$. Thus in order to prove that $(u, \alpha) \perp_{\infty}^a (v, \beta)$, it suffices to prove that $(u_0, \frac{1}{2}) \perp_{\infty} (- u_0, \frac{1}{2})$. In this case, by definition, we have $\Vert (u_0, \frac{1}{2}) \Vert_{\infty} = 1 = \Vert (- u_0, \frac{1}{2}) \Vert_{\infty}$ and $\Vert (u_0, \frac{1}{2}) +  (- u_0, \frac{1}{2}) \Vert_{\infty} = 1$. Thus by \cite[Theorem 3.3]{K14}, we have $(u_0, \frac{1}{2}) \perp_{\infty} (- u_0, \frac{1}{2})$ so that $(u, \alpha) \perp_{\infty}^a (v, \beta)$. 
	
	Finally, we show that for $(u, \alpha), (v, \beta) \in V^{(\cdot)+}$, $(u, \alpha) \perp_{\infty}^a (v, \beta)$ implies $(u, \alpha) \perp (v, \beta)$. Let $(u, \alpha), (v, \beta) \in V^{(\cdot)+}$ with $(u, \alpha) \perp_t^a (v, \beta)$. Without any loss of generality, we assume that $(u, \alpha) \ne (0, 0)$ and $(v, \beta) \ne (0, 0)$. Also, as $\lambda (u, \alpha) \perp_t^a \mu (v, \beta)$ for any $\lambda, \mu \ge 0$, we further assume that $\Vert (u, \alpha) \Vert_0 = \alpha + \Vert u \Vert = 1 = \Vert (v, \beta) \Vert_0 = \beta + \Vert v \Vert$. Since $(u, \alpha), (v, \beta) \in V^{(\cdot)+}$, we have $\Vert u \Vert \le \alpha$ and $\Vert v \Vert \le \beta$. We show that $\Vert u \Vert = \alpha$ and $\Vert v \Vert = \beta$. Assume, to the contrary that $\Vert u \Vert < \alpha$. Then as $(0, 0) \le (u, \Vert u \Vert)$, we get that $(0, 0) \le (0, \alpha - \Vert u \Vert) \le (u, \alpha)$. Thus  $(0, \alpha - \Vert u \Vert) \perp_{\infty} (v, \beta)$ and consequently, $(0, 1) \perp_{\infty} (v, \beta)$. Thus 
	$$1 = \Vert (0, 1) + (v, \beta) \Vert_0 = \Vert v \Vert + \beta + 1 = 2$$
	which is an absurd. Thus $\Vert u \Vert = \alpha = \frac 12$. Similarly,  $\Vert v \Vert = \beta = \frac 12$. Thus as $(u, \alpha) \perp_{\infty} (v, \beta)$, we get 
	$$1 = \Vert (u, \alpha) + (v, \beta) \Vert_0 = \Vert u + v \Vert + 1$$
	and consequently, $v = - u$. Hence $(u, \alpha) = (u, \frac 12) \perp (- u, \frac 12) = (v, \beta)$. This completes the proof.
\end{proof}

\section{Spectral properties of $V^{(\infty)}$}

Recall that in an absolute order unit space $U$, the set of order projections $OP(U)$ is said to \emph{cover $U$ absolutely}, if for each $u \in U$, the exists a unique $p \in OP(U)$ such that $\vert u \vert \le \Vert u \Vert p$ and whenever there exists $q \in OP(U)$ with $\vert u \vert \le \Vert u \Vert q$, then $p \le q$. In this case, $p$ is said to be the \emph{ absolute cover} of $u$ (\cite[Definition 7.1]{K18}). 

Let $(U, e)$ be an absolute order unit space in which $OP(U)$ covers $U$. Given $p \in OP(U)$, the subspace 
$$U_p = \lbrace u \in U: \vert u \vert \le \lambda p ~ \textrm{for some} ~ \lambda > 0 \rbrace$$
of $U$ is an absolute ordered subspace of $U$ such that $(U_p, p)$ is an absolute order unit space. We write $AC(p) := U_p + U_{p'}$ and $AC(p_i; i \in I) := \cap_{i \in I} AC(p_i)$ for any subset $\lbrace p_i: i \in I \rbrace \subset OP(U)$.  Further $c_p^{\pm}(v, \alpha)$ denotes the absolute cover of $(v - \alpha p)^{\pm}$. The following two results generalize the notion of spectral family of a self-adjoint element of a von Neumann algebra.
\begin{theorem}[Spectral Resolution]\label{53} \cite[Theorem 7.10]{K18}
	Let $(U, e)$ be a monotone complete absolute order unit space in which $OP(U)$ covers $U$. Then for every $u \in U$, there exists a unique family $\{ e_{\alpha}: \alpha \in \mathbb{R} \} \subset OP(U)$ such that
	\begin{enumerate}
		\item $\{ e_{\alpha}: \alpha \in \mathbb{R} \}$ is increasing;
		\item $e_{\alpha} = 0$ if $\alpha < - \Vert u \Vert$ and $e_{\alpha} = e$ if $\alpha \ge \Vert u \Vert$;
		\item $u \in AC(e_{\alpha}; \alpha \in \mathbb{R})$;
		\item $C_{e_{\alpha}}(u) \le \alpha e_{\alpha}$ and $C_{e_{\alpha}}^{\prime}(u) \ge \alpha e_{\alpha}^{\prime}$ for each $\alpha \in \mathbb{R}$;
		\item If $u \in AC(p)$ for some $p \in OP(U)$ and if $C_p(u) \le \alpha p$ and $C_p^{\prime}(u) \ge \alpha p^{\prime}$, then $p \le e_{\alpha}$; 
		\item $\wedge_{\alpha > \alpha_0} e_{\alpha} = e_{\alpha_0}$ for each $\alpha_0 \in \mathbb{R}$.
	\end{enumerate}
	In this case, the family of order projections $\{ e_{\alpha}: \alpha \in \mathbb{R} \}$ is called the spectral family of $u$.
\end{theorem}
\begin{theorem}[Spectral Decomposition]\label{54} \cite[Theorem 7.12]{K18}
	Let $(V, e)$ be a monotone complete absolute order unit space in which $OP(V)$ covers $V$ and let $v \in V$.
	Consider the spectral resolution $\{ e_{\alpha} : \alpha \in \mathbb{R} \}$ of $v$ in $OP(V)$. Then for any $\epsilon > 0$ 
	and a finite increasing sequence $\alpha_0  < \dots < \alpha_n$ with $\alpha_0 < - \Vert v \Vert$,$\alpha_n > \Vert v \Vert$ and 
	$\max \{ \alpha_i - \alpha_{i-1} : 1 \le i \le n \} < \epsilon$, we have 
	$$\big\Vert v - \sum_{i=1}^n \alpha_i (e_{\alpha_i} - e_{\alpha_{i-1}}) \big\Vert < \epsilon .$$
\end{theorem}
Let $V$ be a strictly convex Banach space and consider the corresponding absolute order unit space $(V^{(\infty)}, e)$. In this section, we describe the spectral decomposition of an element in $V^{(\infty)}$. Following which we introduce a (non-necessary bilinear) binary operation in it. We use it to characterize spin factors and also to discuss some local Jordan algebra structures in $V^{(\infty)}$. 

First, we show that $OP\left(V^{(\infty)}\right)$ has the covering property. 
\begin{proposition}\label{s1}
	$OP\left(V^{(\infty)}\right)$ covers $V^{(\infty)}$ absolutely. In fact, for each $(v, \alpha) \in V^{(\infty)}$, $(v, \alpha) \ne 0$, we have $\left( \frac{v}{2 \Vert v \Vert}, \frac 12 \right)$ covers $(v, \alpha)$ when $\Vert v \Vert = \alpha$; $\left( \frac{- v}{2 \Vert v \Vert}, \frac 12 \right)$ covers $(v, \alpha)$ when $\Vert v \Vert = - \alpha$ and in other cases, the cover of $(v, \alpha)$ is $e$.
\end{proposition} 
\begin{proof}
	Let $(v, \alpha) \in V^{(\infty)}$. As order projections have the order unit property, without any loss of generality, we may assume that $\Vert v \Vert + \vert \alpha \vert = \Vert (v, \alpha) \Vert_{\infty} = 1$. Let $(v, \alpha) \in V^{(\infty)+}$. Then $\Vert v \Vert \le \alpha$ so that $\alpha \ge \frac 12$. Clearly, $(v, \alpha) \le e$. Let $(v, \alpha) \le (u, \frac 12)$ for some $u \in V$ with $\Vert u \Vert = \frac 12$. Then $\alpha \le \frac 12$ so that $\alpha = \frac 12 = \Vert v \Vert$. Also then $\Vert v - u \Vert \le \frac 12 - \alpha = 0$ so that $v = u$. In other words, when $\Vert v \Vert = \alpha$, then $(u, \frac 12)$ covers $(v, \alpha)$ where $u = \frac{v}{2 \Vert v \Vert}$ and when $\Vert v \Vert < \alpha$, then $e$ covers $(v, \alpha)$. For $(v, \alpha) \in - V^{(\infty)+}$, we consider $- (v, \alpha) \in V^{(\infty)+}$ as $(v, \alpha)$ and $(v, \alpha)$ have the same covers. Finally, let $(v, \alpha) \notin V^{(\infty)+} \cup - V^{(\infty)+}$. Then $\Vert v \Vert > \vert \alpha \vert$ and $\vert (v, \alpha) \vert = \left( \frac{\alpha v}{\Vert v \Vert}, \Vert v \Vert \right)$. Thus as above, $e$ covers $(v, \alpha)$ absolutely.	
\end{proof} 
It is a routine matter to verify that $V^{(\infty)}$ is monotone complete. Below we describe the sptral family of a given element $(v, \alpha)$ of $V^{(\infty)}$.
\begin{proposition}\label{s2}
	The spectral family of $(v, \alpha) \in V^{(\infty)}$ is $\lbrace e_{\lambda}(v, \alpha): \lambda \in \R \rbrace$ where 
	$$e_{\lambda}(v, \alpha) = \begin{cases} 0 & , \mbox{if}~ \lambda < \alpha - \Vert v \Vert \\ 
	\left( \frac{- v}{2 \Vert v \Vert}, \frac 12 \right) & , \mbox{if}~ \alpha - \Vert v \vert \le \lambda < \alpha + \Vert v \Vert \\
	e & , \mbox{if}~ \alpha + \Vert v \Vert \le \lambda. \end{cases}$$
\end{proposition} 
\begin{proof}
	Put $(v, \alpha) - \lambda e = (v, \alpha - \lambda) := v_{\lambda}$. Let us write $C_{\lambda}^+(v, \alpha)$ for the absolute cover of $v_{\lambda}^+$ and $C_{\lambda}^-(v, \alpha)$ for the absolute cover of $v_{\lambda}^-$.
	
	When $\lambda < \alpha - \Vert v \Vert$, $v_{\lambda} \in V^{(\infty)+}$ and by Proposition \ref{s1}, $C_{\lambda}^+(v, \alpha) = e$ and $C_{\lambda}^-(v, \alpha) = 0$. 
	
	When $\lambda = \alpha - \Vert v \Vert$, $v_{\lambda} \in V^{(\infty)+}$ and by Proposition \ref{s1}, $C_{\lambda}^+(v, \alpha) = \left( \frac{v}{2 \Vert v \Vert}, \frac 12 \right)$ and $C_{\lambda}^-(v, \alpha) = 0$. 
	
	When $\lambda = \alpha + \Vert v \Vert$, $v_{\lambda} \in - V^{(\infty)+}$ and by Proposition \ref{s1}, $C_{\lambda}^-(v, \alpha) = \left( \frac{- v}{2 \Vert v \Vert}, \frac 12 \right)$ and $C_{\lambda}^+(v, \alpha) = 0$. 
	
	When $\lambda > \alpha - \Vert v \Vert$, $v_{\lambda} \in - V^{(\infty)+}$ and by Proposition \ref{s1}, $C_{\lambda}^+(v, \alpha) = 0$ and $C_{\lambda}^-(v, \alpha) = e$. 
	
	Now, let $\alpha - \Vert v \Vert < \lambda < \alpha + \Vert v \Vert$. Then $\vert \alpha - \lambda \vert < \Vert v \Vert$ so that $\vert v_{\lambda} \vert = \left( \frac{(\alpha - \lambda) v}{2 \Vert v \Vert}, \Vert v \Vert \right)$. Thus $v_{\lambda}^+ = (\alpha - \lambda + \Vert v \Vert) \left( \frac{v}{2 \Vert v \Vert}, \frac 12 \right)$ and $v_{\lambda}^- = (\alpha - \lambda + \Vert v \Vert) \left( \frac{- v}{2 \Vert v \Vert}, \frac 12 \right)$. Therefore, $C_{\lambda}^+(v, \alpha) = \left( \frac{v}{2 \Vert v \Vert}, \frac 12 \right)$ and $C_{\lambda}^-(v, \alpha) = \left( \frac{- v}{2 \Vert v \Vert}, \frac 12 \right)$. Summarizing, we get 
	$$C_{\lambda}^+(v, \alpha) = \left\{ \begin{array}{ll} e & , \mbox{if}~ \lambda < \alpha - \Vert v \Vert \\ 
	\left( \frac{v}{2 \Vert v \Vert}, \frac 12 \right) & , \mbox{if}~ \alpha - \Vert v \vert \le \lambda < \alpha + \Vert v \Vert \\
	0 & , \mbox{if}~ \alpha + \Vert v \Vert \le \lambda \end{array} \right.$$
	and 
	$$C_{\lambda}^-(v, \alpha) = \left\{ \begin{array}{ll} 0 & , \mbox{if}~ \lambda \le \alpha - \Vert v \Vert \\ 
	\left( \frac{- v}{2 \Vert v \Vert}, \frac 12 \right) & , \mbox{if}~ \alpha - \Vert v \vert < \lambda \le \alpha + \Vert v \Vert \\
	e & , \mbox{if}~ \alpha + \Vert v \Vert < \lambda. \end{array} \right.$$ 
	Since $e_{\lambda}(v, \alpha) = e - C_{\lambda}^+(v, \alpha)$ \cite[Theorem 7.10]{K18}, the result follows.
\end{proof}
\begin{remark}\label{s3} 
	\begin{enumerate}
		\item Let $(v, \alpha) \in V^{(\infty)}$. Then the spectral decomposition of $(v, \alpha)$ is given by 
		$$(v, \alpha) = (\alpha - \Vert v \Vert) \left( \frac{- v}{2 \Vert v \Vert}, \frac 12 \right) + (\alpha + \Vert v \Vert) \left( \frac{v}{2 \Vert v \Vert}, \frac 12 \right).$$ 
		Thus for each $n \in \N$, we can formally define 
		\begin{eqnarray*}
			(v, \alpha)^n &:=& (\alpha - \Vert v \Vert)^n \left( \frac{- v}{2 \Vert v \Vert}, \frac 12 \right) + (\alpha + \Vert v \Vert)^n \left( \frac{v}{2 \Vert v \Vert}, \frac 12 \right) \\ 
			&=& \left( \left( \binom{n}{1} \alpha^{n-1} + \binom{n}{3} \alpha^{n-3} \Vert v \Vert^2 + \cdots \right) v, \left( \binom{n}{0} \alpha^n + \binom{n}{2} \alpha^{n-2} \Vert v \Vert^2 + \cdots \right) \right).
		\end{eqnarray*} 
		In particular, we have $(v, \alpha)^2 = (2 \alpha v, \alpha^2 + \Vert v \Vert^2)$ for any $(v, \alpha) \in V^{(\infty)}$. 
		\item We have $V^{(\infty)+} = \left\lbrace (v, \alpha)^2: (v, \alpha) \in V^{(\infty)} \right\rbrace.$ In fact for any $(v, \alpha) \in V^{(\infty)}$, we have $(v, \alpha)^2 = (2 \alpha v, \alpha^2 + \Vert v \Vert^2) \in V^{(\infty)}$ for $\Vert 2 \alpha v \Vert = 2 \vert \alpha \vert \Vert v \Vert \le \alpha^2 + \Vert v \Vert^2$. Conversely, let 
		$(v, \alpha) \in V^{(\infty)+}$ so that $\Vert v \Vert \le \alpha$. If $v = 0$, then $(0, \sqrt{\alpha})$ the unique positive element in $V^{(\infty)+}$ such that $(0, \sqrt{\alpha})^2 = (v, \alpha)$. So we assume that $v \ne 0$. Put 
		\[ \lambda = \dfrac{\sqrt{\alpha - \sqrt{\alpha^2 - \Vert v \Vert^2}}}{\Vert v \Vert \sqrt{2}} \quad \textrm{and} \quad \mu = \dfrac{\Vert v \Vert}{\sqrt{2 (\alpha - \sqrt{\alpha^2 - \Vert v \Vert^2})}}. \] 
		Then $0 \le \lambda \Vert v \Vert \le \mu$ and $(\lambda v, \mu)$ is the unique positive element in $V^{(\infty)+}$ such that $(\lambda v, \mu)^2 = (v, \alpha)$. In particular, every positive element in $V^{(\infty)}$ has a unique positive square root in $V^{(\infty)}$.
		\item Consider $(u, \alpha), (v, \beta) \in V^{(\infty)}$. Then for $\lambda, \mu \in \R$, we have 
		\[ \left( \lambda (u, \alpha) + \mu (v, \beta) \right)^2 = \left(2 (\lambda \alpha + \mu \beta) (\lambda u + \mu v), (\lambda \alpha + \mu \beta)^2 + \Vert \lambda u + \mu v \Vert^2 \right). \] 
		In this expression, we can segregate terms with $\lambda^2$, $\lambda \mu$ and $\mu^2$ explicitly, provided that $\Vert \lambda u + \mu v \Vert_0^2$ can be expanded in the said form. This is possible in few cases: 
		\begin{enumerate}
			\item \emph{$V$ is a Hilbert space}: In this case, $\Vert \lambda u + \mu v \Vert^2 = \lambda \Vert u \Vert^2 + 2 \lambda \mu \langle u, v \rangle + \mu^2 \Vert v \Vert^2$. Thus 
			\begin{eqnarray*}
				&\left( \lambda (u, \alpha) + \mu (v, \beta) \right)^2 = \lambda^2 (u, \alpha)^2 + 2 \lambda \mu (\beta u + \alpha v, \alpha \beta + \langle u, v \rangle ) + \mu^2 (v, \beta)^2 \\
				&= \lambda^2 (u, \alpha)^2 + 2 \lambda \mu (\beta u + \alpha v, \alpha \beta + \frac 14 \lbrace \Vert u + v \Vert^2 - \Vert u - v \Vert^2 \rbrace ) + \mu^2 (v, \beta)^2.
			\end{eqnarray*}
			\item \emph{The set $\lbrace u, v \rbrace$ is linearly dependent}: In this case, without any loss of generality, we may assume that $u = k u_0$ and $v = l u_0$ for some $u_0 \in V_0$ with $\Vert u_0 \Vert_0 = 1$ and $k, l \in \R$. Then 
			\begin{eqnarray*}
				&\left( \lambda (k u_0, \alpha) + \mu (l u_0, \beta) \right)^2 = \left(2 (\lambda \alpha + \mu \beta) (\lambda k + \mu l) u_0, (\lambda \alpha + \mu \beta)^2 +  (\lambda k + \mu l)^2 \right) \\
				&= \lambda^2 (2 k \alpha u_0, \alpha^2 + k^2) + 2 \lambda \mu ((k \beta + l \alpha) u_0, \alpha \beta + k l) + \mu^2 (2 l \beta u_0, \beta^2 + l^2) \\
				&= \lambda^2 (k u_0, \alpha)^2 + 2 \lambda \mu ((k \beta + l \alpha) u_0, \alpha \beta + k l) + \mu^2 (l u_0, \beta)^2.
			\end{eqnarray*}
		\end{enumerate} 
	\end{enumerate} 
\end{remark} 

\subsection{The binary operation} 
The two cases described in Remark \ref{s3} are not the exhaustive. Nevertheless, these observations encourage us to introduce a binary operation $\circ$ in $V^{(\infty)}$ in the following sense. For $(u, \alpha), (v, \beta) \in V^{(\infty)}$, we define 
\begin{eqnarray*}
	(u, \alpha) \circ (v, \beta) &:=& \frac 14 \left\lbrace (u + v, \alpha + \beta)^2 - (u - v, \alpha - \beta)^2 \right\rbrace \\ 
	&=& \left(\alpha v + \beta u, \alpha \beta + \frac 14 (\Vert u + v \Vert^2 - \Vert u - v \Vert^2) \right).
\end{eqnarray*} 
However, in general, $\circ$ may not be bilinear. 
\begin{theorem}\label{s4}
	The binary operation $\circ$ is bilinear in $V^{(\infty)}$ if and only if $V$ is a Hilbert space.
\end{theorem}
\begin{proof}
	When $V$ is a Hilbert space, $V^{(\infty)}$ is a $JB$-algebra with $\circ$ as the corresponding Jordan product \cite[Lemma 6.1.3]{H-OS84}. Conversely, we assume that $\circ$ is bilinear. Then for $u, v \in V$ we have 
	\[ (2 u + v, 0) \circ (v, 0) = 2 (u, 0) \circ (v, 0) + (v, 0) \circ (v, 0) \] 
	so that 
	\[ \Vert 2 u + 2 v \Vert^2 - \Vert 2 u \Vert^2 = 2 \Vert u + v \Vert^2 - 2 \Vert u - v \Vert^2 + 4 \Vert v \Vert^2. \] 
	Thus 
	\[ \Vert u + v \Vert^2 + \Vert u - v \Vert^2 = 2 \Vert u \Vert^2 + 2 \Vert v \Vert^2 \] 
	for all $u, v \in V$. Hence $V$ is a Hilbert space.
\end{proof}
Next, we discuss some general properties of the binary operation $\circ$. 
We begin with a relation between the absolute orthogonality among a pair of elements in $V^{(\infty)}$ and their corresponding $\circ$-product. We use the following result.
\begin{lemma}\label{s7}
	Let $V$ be a strictly convex Banach space and assume that $u, v \in V$ and $\alpha, \beta \in \R$. Then $\vert (u, \alpha) \vert \perp \vert (v, \beta) \vert$ if and only if $\beta u + \alpha v = 0$, $\Vert u \Vert = \vert \alpha \vert$ and $\Vert v \Vert = \vert \beta \vert$.
\end{lemma} 
\begin{proof}
	First, we assume that $\vert (u, \alpha) \vert \perp \vert (v, \beta) \vert$.
	
	Case 1. $(u, \alpha), (v, \beta) \in V^{(\infty)+}$. In this case, $(u, \alpha) \perp (v, \beta)$. Thus there exists $u_0 \in V$ with $\Vert u_0 \Vert = 1$ such that $(u, \alpha) = \alpha (u_0, 1)$ and $(v, \beta) = \beta (- u_0, 1)$. Therefore, $\beta u + \alpha v = 0$, $\Vert u \Vert = \vert \alpha \vert$ and $\Vert v \Vert = \vert \beta \vert$. The case $- (u, \alpha), - (v, \beta) \in V^{(\infty)+}$ is similar.
	
	Case 2. $(u, \alpha), - (v, \beta) \in V^{(\infty)+}$. In this case, $(u, \alpha) \perp (- v, - \beta)$. Thus there exists $u_0 \in V$ with $\Vert u_0 \Vert = 1$ such that $(u, \alpha) = \alpha (u_0, 1)$ and $(v, \beta) = \beta (- u_0, 1)$. Therefore, once more, we have $\beta u + \alpha v = 0$, $\Vert u \Vert = \vert \alpha \vert$ and $\Vert v \Vert = \vert \beta \vert$. The case $- (u, \alpha), (v, \beta) \in V^{(\infty)+}$ is similar. 
	
	Case 3. $(u, \alpha), (v, \beta) \notin V^{(\infty)+} \cup V^{(\infty)+}$. In this case, we have $\vert \alpha \vert < \Vert u \Vert$, $\vert \beta \vert < \Vert v \Vert$, $\vert (u, \alpha) \vert = \left( \frac{\alpha u}{\Vert u \Vert}, \Vert u \Vert \right)$ and $\vert (v, \beta) \vert = \left( \frac{\beta v}{\Vert v \Vert}, \Vert v \Vert \right)$. Now, by Proposition \ref{02}, $\vert (u, \alpha) \vert \not\perp \vert (v, \beta) \vert$. Thus Case 3 does not arise.
	
	Conversely, we assume that $\beta u + \alpha v = 0$, $\Vert u \Vert_0 = \vert \alpha \vert$ and $\Vert v \Vert_0 = \vert \beta \vert$. If $\alpha = 0$, the $u = 0$ so that $\vert (u, \alpha) \vert (= (0, 0)) \perp \vert (v, \beta) \vert$. Similarly, $\beta = 0$ implies $\vert (u, \alpha) \vert \perp \vert (v, \beta) \vert$. So we may assume that $\alpha \beta \ne 0$. Then $u \ne 0$ and $v \ne 0$. Also, then $u_0 := \alpha^{-1} u = - \beta^{-1} v$ has norm one and we have $(u, \alpha) = \alpha (u_0, 1)$ and $(v, \beta) = \beta (- u_0, 1)$. Now, by Proposition \ref{02}, $\vert (u, \alpha) \vert \perp \vert (v, \beta) \vert$. 
\end{proof}
\begin{theorem}\label{s8}
	Let $V$ be a strictly convex Banach space and assume that $u, v \in V$ and $\alpha, \beta \in \R$. Then 
	\begin{enumerate}[$(i)$]
		\item $(u, \alpha) \circ (v, \beta) = (0, 0)$ and $\lbrace u, v \rbrace$ is linearly independent if and only if $u \ne 0$, $v \ne 0$, $\alpha = 0 = \beta$ and $\Vert u + v \Vert_0 = \Vert u - v \Vert_0$; 
		\item $(u, \alpha) \circ (v, \beta) = (0, 0)$ and $\lbrace u, v \rbrace$ is linearly dependent if and only if $\vert (u, \alpha) \vert \perp \vert (v, \beta) \vert$.
	\end{enumerate}
\end{theorem}
\begin{proof}
	We note that $(u, \alpha) \circ (v, \beta) = (0, 0)$ if and only if $\beta u + \alpha v = 0$ and $\Vert u - v \Vert^2 - \Vert u + v \Vert^2 = 4 \alpha \beta$. 
	
	(i). If $(u, \alpha) \circ (v, \beta) = (0, 0)$ and if $\lbrace u, v \rbrace$ is linearly independent, then $u \ne 0$, $v \ne 0$, $\alpha = 0 = \beta$ and $\Vert u + v \Vert = \Vert u - v \Vert$. Conversely, assume that $u \ne 0$, $v \ne 0$, $\alpha = 0 = \beta$ and $\Vert u + v \Vert = \Vert u - v \Vert$. Then 
	$$(u, 0) \circ (v, 0) = (0, \Vert u + v \Vert^2 - \Vert u - v \Vert^2) = (0, 0).$$
	We show that $\lbrace u, v \rbrace$ is linearly independent. Let $\lambda u + \mu v = 0$. If $\lambda \ne 0$, then $u = - \lambda^{-1} \mu v$. Thus $\Vert u + v \Vert = \Vert u - v \Vert$ implies that $\vert 1 - \lambda^{-1} \mu \vert \Vert v \Vert = \vert 1 + \lambda^{-1} \mu \vert \Vert v \Vert$. As $v \ne 0$, we conclude that $\mu = 0$ so that $u = 0$. But by assumption, $u \ne 0$ so $\lambda$ must be zero. Similarly, we can show that $\mu = 0$. Therefore, $\lbrace u, v \rbrace$ is linearly independent. 
	
	(ii). Next, we assume that $(u, \alpha) \circ (v, \beta) = (0, 0)$ and that $\lbrace u, v \rbrace$ is linearly dependent. First let $u = 0$. In this case, $\alpha v = 0$. If $\alpha = 0$, then clearly, $\vert (u, \alpha) \vert (= (0, 0)) \perp \vert (v, \beta) \vert$. So let $\alpha \ne 0$. Then $v = 0$. Now, $\Vert u - v \Vert^2 - \Vert u + v \Vert^2 = 4 \alpha \beta$ yields that $\beta = 0$. Therefore, $\vert (v, \beta) \vert \perp (= (0, 0)) \perp \vert (u, \alpha) \vert$. The case $v = 0$ is similar. In the same way, we can show that $\vert (u, \alpha) \vert \perp \vert (v, \beta) \vert$ whenever $\alpha = 0$ or $\beta = 0$. So we may assume that $u \ne 0$, $v \ne 0$ and $\alpha \beta \ne 0$. Put $u_0 := \alpha^{-1} u = - \beta^{-1} v$. Then 
	\[ 4 \alpha \beta = \Vert (\alpha + \beta) u_0 \Vert^2 - \Vert (\alpha - \beta) u_0 \Vert^2 = 4 \alpha \beta \Vert u_0 \Vert^2 \] 
	so that $\Vert u_0 \Vert = 1$ as $\alpha \beta \ne 0$. Hence by Proposition \ref{02}, we have $\vert (u, \alpha) \vert \perp \vert (v, \beta) \vert$. 
	
	Finally, we assume that $\vert (u, \alpha) \vert \perp \vert (v, \beta) \vert$. Then by Lemma \ref{s7}, we have $\beta u + \alpha v = 0$, $\Vert u \Vert = \vert \alpha \vert$ and $\Vert v \Vert = \vert \beta \vert$. If $\alpha = 0$, then $u = 0$ so that $u$ is linearly dependent on $v$ and we have $(u, \alpha) \circ (v, \beta) = (0, 0)$. So we assume that $\alpha \ne 0$. Then $\lbrace u, v \rbrace$ is linearly dependent with $v = - \alpha^{-1} \beta u$. Thus 
	\[ \Vert u - v \Vert^2 - \Vert u + v \Vert^2 = (1 + \alpha^{-1} \beta)^2 \Vert u \Vert^2 - (1 - \alpha^{-1} \beta)^2 \Vert u \Vert^2 = 4 \alpha \beta \] 
	as $\Vert u \Vert = \vert \alpha \vert$. Thus $(u, \alpha) \circ (v, \beta) = (0, 0)$. This completes the proof.
\end{proof}
\begin{remark}
	We produce an example to show that in Theorem \ref{s8}(i), $(u, \alpha) \circ (v, \beta) = (0, 0)$ may not imply $(k (u, \alpha)) \circ (l (v, \beta)) = (0, 0)$ for all $k, l \in \R$. Consider $V_0 = \ell_4^2$ and let $u = \left(1, \frac{\sqrt{3} + \sqrt{5}}{2} \right)$ and $v = \left(1, \frac{\sqrt{3} - \sqrt{5}}{2} \right)$. Then $\Vert u + k v \Vert_4 = \Vert u -  k v \Vert_4$ with $k > 0$ if and only if $k = 1$. Thus $(k (u, 0)) \circ (l (v, 0)) = (0, 0)$ with $k, l \in \R$ if and only if $\vert k l \vert = 1$.
	
	However, the case (ii) of Theorem \ref{s8} is more liberal. If $\lbrace u, v \rbrace$ is linearly dependent with $(u, \alpha) \circ (v, \beta) = (0, 0)$, then $(k (u, \alpha)) \circ (l (v, \beta)) = (0, 0)$ for any $k, l \in \R$. This follows from the fact that $\vert (u, \alpha) \vert \perp \vert (v, \beta) \vert$ implies $\vert k (u, \alpha) \vert \perp \vert l (v, \beta) \vert$ for all $k,l \in \R$. In fact, we can say more.
\end{remark}
\begin{proposition}\label{s5}
	Let $V$ be a strictly convex Banach space. For $u \in V$ with $\Vert u \Vert = 1$, consider 
	\[ V(u) := \lbrace (\alpha u, \beta): \alpha, \beta \in \R \rbrace. \] 
	Then $V(u)$ with $\circ$ (restricted to $V(u)$) is an associative commutative algebra over $\R$. Further, when the norm $\Vert \cdot \Vert_{\infty}$ of $V^{(\infty)}$ is restricted to $V(u)$, then $V(u)$ becomes a unital associative $JB$-algebra with the order unit $e = (0, 1)$ as the unity.
\end{proposition}
\begin{proof}
	For $\alpha, \beta, \gamma, \delta \in \mathbb{R}$ we have 
	$$(\alpha u, \beta) \circ (\gamma u, \delta) = ((\alpha \delta + \beta \gamma) u, \alpha \gamma + \beta \delta).$$ 
	Now, it is straight forward to show that $\circ$ is bilinear, commutative and associative in $V(u)$ so that $V(u)$ is a commutative, associative algebra over real field. Further, we have $e = (0, 1)$ is the multiplicative identity. Next, we show that 
	\[ V(u)^+ := V(u) \cap V^{(\infty)+} = \lbrace (\alpha u, \beta)^2: \alpha, \beta \in \R \rbrace. \] 
	Note that 
	$$(\alpha u, \beta)^2 = (2 \alpha \beta u, \Vert \alpha u \Vert^2 + \beta^2) = (2 \alpha \beta u, \alpha^2 + \beta^2)$$ 
	for all $\alpha, \beta \in \R$. Since $\Vert 2 \alpha \beta u \Vert = 2 \vert \alpha \beta \vert \le \alpha^2 + \beta^2$, we get that $(\alpha u, \beta)^2 \in V^{(\infty)+}$ for all $\alpha, \beta \in \R$. Conversely, let $(\alpha u, \beta) \in V^{(\infty)+}$. Then $\vert \alpha \vert = \Vert \alpha u \Vert \le \beta$. Put 
	\[ \lambda = sign(\alpha) \sqrt{\frac{\beta - \sqrt{\beta^2 - \alpha^2}}{2}} \quad \textrm{and} \quad \mu = \sqrt{\frac{\beta + \sqrt{\beta^2 - \alpha^2}}{2}}. \] 
	Then $(\lambda u, \mu)$ is positive and $(\lambda u, \mu)^2 = (\alpha u, \beta)$. 
	
	Let $\alpha, \beta \in \R$. Then 
	\begin{eqnarray*}
		\Vert (\alpha u, \beta)^2 \Vert_{\infty} &=& \Vert (2 \alpha \beta u, \alpha^2 + \beta^2) \Vert_{\infty} \\ 
		&=& \Vert 2 \alpha \beta u \Vert + \alpha^2 + \beta^2 \\ 
		&=& (\vert \alpha \vert + \vert \beta \vert)^2 \\ 
		&=& \Vert (\alpha u, \beta) \Vert_{\infty}^2.
	\end{eqnarray*}
	In particular, we have $\Vert (\alpha u, \beta) \Vert_{\infty} \le 1$ if and only if $0 \le (\alpha u, \beta)^2 \le e$. Now by \cite[Proposition 3.1.6]{H-OS84}, $V(u)$ is a unital $JB$-algebra with the restricted norm $\Vert \cdot \Vert_{\infty}$.
\end{proof} 
We can generalize Proposition \ref{s5}. Let $x, y$ be any two elements in a normed linear space $X$. We say that $x \perp_2 y$, if $\Vert x + k y \Vert^2 = \Vert x \Vert^2 + \Vert k y \Vert^2$ for all $k \in \R$ \cite[Section 2]{K14}. 
\begin{theorem}\label{s6}
	Let $V$ be a strictly convex Banach space and assume that $u, v \in V$ with $\Vert u \Vert = 1 = \Vert v \Vert$ such that $u \perp_2 v$. Consider 
	$$V(u, v) := \left\lbrace \left( \alpha_1 u + \alpha_2 v, \beta \right): \alpha_1, \alpha_2, \beta \in \R \right\rbrace.$$
	Then $V(u, v)$ together with $\circ$ (restricted to $V(u, v)$) is a (non-associative) Jordan algebra. Further, when the norm $\Vert \cdot \Vert_{\infty}$ of $V^{(\infty)}$ is restricted to $V(u, v)$, then $V(u, v)$ becomes a unital $JB$-algebra with the order unit $e = (0, 1)$ as the unity.
\end{theorem} 
\begin{proof}
	Let $\alpha_1, \alpha_2, \beta, \gamma_1, \gamma_2, \delta \in \R$. Then 
	\begin{eqnarray*}
		\Vert \alpha_1 u + \alpha_2 v + \gamma_1 u + \gamma_2 v \Vert^2 &=& \Vert (\alpha_1 + \gamma_1) u \Vert^2 + \Vert (\alpha_2 + \gamma_2) v \Vert^2 \\ 
		&=& (\alpha_1 + \gamma_1)^2 + (\alpha_2 + \gamma_2)^2 
	\end{eqnarray*} 
	and 
	\begin{eqnarray*}
		\Vert \alpha_1 u + \alpha_2 v - \gamma_1 u - \gamma_2 v \Vert^2 &=& \Vert (\alpha_1 - \gamma_1) u \Vert^2 + \Vert (\alpha_2 - \gamma_2) v \Vert^2 \\ 
		&=& (\alpha_1 - \gamma_1)^2 + (\alpha_2 - \gamma_2)^2. 
	\end{eqnarray*}
	Thus 
	\begin{eqnarray*}
		(\alpha_1 u + \alpha_2 v, \beta) &\circ& (\gamma_1 u + \gamma_2 v, \delta) \\ 
		&=& (\delta (\alpha_1 u + \alpha_2 v) + \beta (\gamma_1 u + \gamma_2 v), \\ 
		& &\beta \delta + \frac 14 (\Vert \alpha_1 u + \alpha_2 v + \gamma_1 u + \gamma_2 v \Vert^2 - \Vert \alpha_1 u + \alpha_2 v - \gamma_1 u - \gamma_2 v \Vert^2)) \\
		&=& ((\delta \alpha_1 + \beta \gamma_1) u + (\delta \alpha_2 + \beta \gamma_2) v, \alpha_1 \gamma_1 + \alpha_2 \gamma_2 + \beta \delta ).
	\end{eqnarray*}
	It is a tedious but straight forward calculation to show that $\circ$ is a bilinear and commutative product in $V(u, v)$ satisfying Jordan identity: 
	\[ x \circ (y \circ x^2) = (x \circ y) \circ x^2 \] 
	for all $x, y \in V(u, v)$. However, as 
	$$(v, 0) \circ (v, 0) = (0, 1)$$ 
	and 
	$$(u, 0) \circ (v, 0) = (0, 0),$$ 
	we have  
	$$(u, 0) \circ ((v, 0) \circ (v, 0)) \ne ((u, 0) \circ (v, 0)) \circ (v, 0) .$$ 
	Thus $\circ$ is not associative in $V(u, v)$. 
	
	Next, we note that 
	\[ \Vert (\alpha_1 u + \alpha_2 v, \beta) \Vert_{\infty} = \Vert \alpha_1 u + \alpha_2 v \Vert + \vert \beta \vert = (\alpha_1^2 + \alpha_2^2)^{\frac 12} + \vert \beta \vert \] 
	for all $\alpha_1, \alpha_2, \beta \in \R$. Also 
	$$(\alpha_1 u + \alpha_2 v, \beta)^2 = (2 \alpha_1 \beta u + 2 \alpha_2 \beta v, \alpha_1^2 + \alpha_2^2 + \beta^2)$$ 
	so that 
	\[ \Vert (\alpha_1 u + \alpha_2 v, \beta)^2 \Vert_{\infty} = 2 \vert \beta \vert (\alpha_1^2 + \alpha_2^2)^{\frac 12} + \alpha_1^2 + \alpha_2^2 + \beta^2 = \Vert (\alpha_1 u + \alpha_2 v, \beta) \Vert_{\infty}^2. \] 
	Since $\Vert \cdot \Vert_{\infty}$ is an order unit norm on $V(u, v)$ with the order unit $e = (0, 1)$ as distinguished multiplicative identity and since $x^2 \in V(u, v)^+$ for all $x \in V(u, v)$, we conclude that 
	\[ - e \le (\alpha_1 u + \alpha_2 v, \beta) \le e \quad \textrm{if and only if} \quad 0 \le (\alpha_1 u + \alpha_2 v, \beta)^2 \le e. \] 
	Hence by \cite[Proposition 3.1.6]{H-OS84}, $V(u, v)$ is a (non-associative) $JB$-algebra.
\end{proof} 
\begin{example}
	Consider the $3$-dimensional real space $\ell_4^3$. Let $u = 2^{- \frac 14} (1, 1, 0)$ and $v = 18^{- \frac 14} (1, - 1, 2)$. Then $\Vert u \Vert_4 = 1 = \Vert v \Vert_4$ with $u \perp_2 v$. Let $H_1$ be the linear span of $u$ and $v$. Then  
	\[ H_1 = \lbrace (\alpha, \beta, \gamma): \gamma = \alpha + \beta \rbrace. \] 
	We note that $H_1$ is a Hilbert space with the inner product 
	\[ \langle (\alpha, \beta, \alpha + \beta), (\gamma, \delta, \gamma + \delta) \rangle = \frac{1}{\sqrt{2}} \left\lbrace 2 (\alpha \gamma + \beta \delta) + (\alpha \delta + \beta \gamma) \right\rbrace \] 
	in such a way that the corresponding norm coincides with $\Vert \cdot \Vert_4$ restricted to $H_1$. We note that $H_1$ is linearly isomorphic to the plane $z = x + y$.
\end{example} 
In general, there is no guaranty that a strictly convex Banach space may have pair of $2$-orthogonal pair of non-zero vectors. 
\begin{theorem}
	Let $1 < p < \infty$, $p \ne 2$ and consider $u = (\alpha_1, \alpha_2), v = (\beta_1, \beta_2) \in \ell_p^2$. Then $u \perp_2 v$ if and only if either $u = 0$ or $v = 0$.
\end{theorem}
\begin{proof}
	Let us assume that $u \perp_2 v$ with $ u \ne 0$ and $v \ne 0$. Without any loss of generality, we may assume that $\Vert u \Vert_p = 1 = \Vert v \Vert_p$. That is, 
	\[ \vert \alpha_1 \vert^p + \vert \alpha_2 \vert^p = 1 = \vert \beta_1 \vert^p + \vert \beta_2 \vert^p. \] 
	As $u \perp_2 v$, we have 
	\[ \Vert k u + l v \Vert_p^2 = k^2 \Vert u \Vert_p^2 + l^2 \Vert v \Vert_p^2 = k^2 + l^2 \] for all $k, l \in \R$. 
	In other words, 
	\[ (\dagger) \qquad (\vert k \alpha_1 + l \beta_1 \vert^p + \vert k \alpha_2 + l \beta_2 \vert^p)^{\frac 2p} = k^2 + l^2 \] 
	for all $k, l \in \R$. For $k = \alpha_1$ and $l = - \beta_1$, we get 
	\[ (*) \qquad (\beta_1 \alpha_2 - \alpha_1 \beta_2)^2 = \beta_1^2 + \alpha_1^2 \] 
	and for $k = \alpha_2$ and $l = - \beta_2$, we get 
	\[ (**) \qquad (\beta_2 \alpha_1 - \alpha_2 \beta_1)^2 = \beta_2^2 + \alpha_2^2. \] 
	In other words, 
	\[ (\beta_2 \alpha_1 - \alpha_2 \beta_1)^2 = \alpha_1^2 + \beta_1^2 = \alpha_2^2 + \beta_2^2. \] 
	Since $\vert \alpha_2 \vert = (1 - \vert \alpha_1 \vert^p)^{\frac 1p}$ and $\vert \beta_1 \vert = (1 - \vert \beta_2 \vert^p)^{\frac 1p}$, we deduce that 
	\[ \vert \alpha_1 \vert^2 - (1 - \vert \alpha_1 \vert^p)^{\frac 2p} = \vert \beta_2 \vert - (1 - \vert \beta_2 \vert^p)^{\frac 2p}. \] 
	Consider the real valued function $f(x) = x^2 - (1 - x^p)^{\frac 2p}$, $x \in [0, 1]$. Then $f'(x) = 2 x + 2 x^{p - 1} (1 - x^p)^{\frac 2p - 1} > 0$ for all $x \in (0, 1)$ so that $f$ is strictly increasing in $[0, 1]$. Thus $\vert \alpha_1 \vert = \vert \beta_2 \vert$ and consequently, $\vert \alpha_2 \vert = \vert \beta_1 \vert$. 
	
	As $u \ne 0$, we have either $\alpha_1 \ne 0$ or $\alpha_2 \ne 0$. For definiteness, we assume that $\alpha_1 \ne 0$. Then $\beta_2 \ne 0$. We show that $\alpha_2 = 0 = \beta_1$.
	
	Case 1. $\alpha_1 \alpha_2 \beta_1 \beta_2 \ge 0$. 
	
	Replacing $l$ by $- l$ in $(\dagger)$, we get 
	\[ (\vert k \alpha_1 + l \beta_1 \vert^p + \vert k \alpha_2 + l \beta_2 \vert^p)^{\frac 2p} = (\vert k \alpha_1 - l \beta_1 \vert^p + \vert k \alpha_2 - l \beta_2 \vert^p)^{\frac 2p} \] 
	for all $k, l \in \R$. In particular, for $k = \beta_1$ and $l = \alpha_1$ we get 
	\[ \vert 2 \alpha_1 \beta_1 \vert^p + \vert \beta_1 \alpha_2 + \alpha_1 \beta_2 \vert^p = \vert \beta_1 \alpha_2 - \alpha_1 \beta_2 \vert^p. \] 
	Dividing this identity by $\vert \alpha_1 \beta_2 \vert^p$, we get 
	\[ 2^p \left\vert \frac{\beta_1}{\beta_2} \right\vert^p + \left\vert \frac{\beta_1 \alpha_2}{\alpha_1 \beta_2} + 1 \right\vert^p = \left\vert \frac{\beta_1 \alpha_2}{\alpha_1 \beta_2} - 1 \right\vert^p. \] 
	As $\alpha_1 \alpha_2 \beta_1 \beta_2 \ge 0$, we have $0 \le \frac{\beta_1 \alpha_2}{\alpha_1 \beta_2}= \left\vert \frac{\beta_1 \alpha_2}{\alpha_1 \beta_2} \right\vert = \left\vert \frac{\alpha_2}{\alpha_1} \right\vert^2$ so that 
	\[ 2^p \left\vert \frac{\alpha_2}{\alpha_1} \right\vert^p + \left\vert \left\vert \frac{\alpha_2}{\alpha_1} \right\vert^2 + 1 \right\vert^p = \left\vert \left\vert \frac{\alpha_2}{\alpha_1} \right\vert^2 - 1 \right\vert^p. \] 
	Now, it follows that $\left\vert \frac{\alpha_2}{\alpha_1} \right\vert^2 = 0$ so that $\alpha_2 = 0$. 
	
	Case 2. $\alpha_1 \alpha_2 \beta_1 \beta_2 \le 0$. 
	
	Since $\vert \alpha_1 \vert^p + \vert \alpha_2 \vert^p = 1 = \vert \beta_1 \vert^p + \vert \beta_2 \vert^p$, $(*)$ may be homogenized to 
	\[ (\beta_1 \alpha_2 - \alpha_1 \beta_2)^2 = \beta_1^2 + \alpha_1^2 = \beta_1^2 (\vert \alpha_1 \vert^p + \vert \alpha_2 \vert^p)^{\frac 2p} + \alpha_1^2 (\vert \beta_1 \vert^p + \vert \beta_2 \vert^p)^{\frac 2p}. \] 
	Dividing this identity by $\vert \alpha_1 \beta_2 \vert^2$, we get 
	\[ (\#) \qquad \left( \frac{\beta_1 \alpha_2}{\alpha_1 \beta_2} - 1 \right)^2 = \left( \frac{\beta_1}{\beta_2} \right)^2 \left( 1 + \left\vert \frac{\alpha_2}{\alpha_1} \right\vert^p \right)^{\frac 2p} + \left( 1 + \left\vert \frac{\beta_1}{\beta_2} \right\vert^p \right)^{\frac 2p}. \] 
	Now $\left\vert \frac{\alpha_2}{\alpha_1} \right\vert = \left\vert \frac{\beta_1}{\beta_2} \right\vert$. Thus as $\alpha_1 \alpha_2 \beta_1 \beta_2 \le 0$, we have $0 \ge \frac{\beta_1 \alpha_2}{\alpha_1 \beta_2} = - \left\vert \frac{\alpha_2}{\alpha_1} \right\vert^2$. Hence $(\#)$ reduces to 
	\[ \left( 1 + \left\vert \frac{\alpha_2}{\alpha_1} \right\vert^2 \right)^2 = \left( 1 + \left\vert \frac{\alpha_2}{\alpha_1} \right\vert^2 \right) \left( 1 + \left\vert \frac{\alpha_2}{\alpha_1} \right\vert^p \right)^{\frac 2p} \] 
	so that $ 1 + \left\vert \frac{\alpha_2}{\alpha_1} \right\vert^2 = \left( 1 + \left\vert \frac{\alpha_2}{\alpha_1} \right\vert^p \right)^{\frac 2p}$. Therefore, $\left( 1 + \left\vert \frac{\alpha_2}{\alpha_1} \right\vert^2 \right)^{\frac 12} = \left( 1 + \left\vert \frac{\alpha_2}{\alpha_1} \right\vert^p \right)^{\frac 1p}$. Since $p \ne 2$, we deduce that $\left\vert \frac{\alpha_2}{\alpha_1} \right\vert = 0$ whence $\alpha_2 = 0$. 
	
	Hence in both cases, we get that $\alpha_1 \ne 0$ implies $\alpha_2 = 0$ so that $\vert \alpha_1 \vert = \vert \beta_2 \vert = 1$ and $\alpha_2 = 0 = \beta_1$. Now, by $(\dagger)$, we have $k^2 = k^2 + l^2$ for all $k, l \in \R$, which is absurd. Hence for $u \perp_2 v$, we must have either $u = 0$ or $v = 0$. The converse holds trivially.
\end{proof}

\thanks{{\bf Acknowledgements:} The author is thankful to the referee for his valuable suggestions. This research was partially supported by Science and Engineering Research Board, Department of Science and Technology, Government of India sponsored  Mathematical Research Impact Centric Support Project (Reference No. MTR/2020/000017).}

\end{document}